\documentclass[11pt,reqno]{amsart}
\pdfoutput=1

\usepackage[english]{babel}
\usepackage[utf8]{inputenc}
\usepackage[T1]{fontenc}
\usepackage{lmodern} \normalfont 
\DeclareFontShape{T1}{lmr}{bx}{sc} { <-> ssub * cmr/bx/sc }{}
\usepackage{microtype}	

\usepackage{amssymb}
\usepackage{amsmath}
\usepackage{amsthm}
\usepackage{mathtools}
\usepackage{mathrsfs}
\usepackage{accents} 
\usepackage{etoolbox}
\usepackage{siunitx}
\usepackage{dsfont} 
\usepackage{graphicx}
\usepackage{color}
\usepackage[dvipsnames]{xcolor}
\usepackage{circuitikz}
\usepackage{tikz}
\usetikzlibrary{calc,positioning,shapes}
\usetikzlibrary{patterns,decorations.pathmorphing,decorations.markings}
\usepackage{pgfplots}
\pgfplotsset{compat=newest}
\usepackage[margin=10pt,font=small,labelfont=bf,labelsep=endash]{caption}
\usepackage{subcaption}

\usepackage{booktabs}
\usepackage{paralist}
\usepackage{soul}

\numberwithin{equation}{section}

\usepackage{algorithm}
\usepackage{algorithmicx}
\usepackage{algpseudocode}

\usepackage{cite}
\usepackage{multirow} 
\usepackage{enumitem} 
\setlist[enumerate]{label=(\roman*)}
\usepackage{xspace}

\usepackage{hyperref}
\hypersetup{colorlinks,linkcolor=teal,citecolor=teal,urlcolor=teal}
\usepackage[nameinlink,noabbrev]{cleveref}

\theoremstyle{plain}
\newtheorem{theorem}{Theorem}[section]

\newtheorem{lemma}[theorem]{Lemma}
\Crefname{lemma}{Lemma}{Lemmas}

\newtheorem{remark}[theorem]{Remark}
\newtheorem{definition}[theorem]{Definition}

\newtheorem{example}[theorem]{Example}
\newtheorem{problem}[theorem]{Problem}

\newcommand{\N}{\ensuremath\mathbb{N}}

\newcommand{\R}{\ensuremath\mathbb{R}}

\newcommand{\GL}[2]{\mathrm{GL}_{#1}(#2)}
\newcommand{\stiefelM}[2]{\mathrm{St}(#1,#2)}
\newcommand{\Spd}[1]{\mathcal{S}^{#1}_{\succ}}
\newcommand{\Spsd}[1]{\mathcal{S}^{#1}_{\succeq}}
\newenvironment{smallbmatrix}{\left[\begin{smallmatrix}}{\end{smallmatrix}\right]}

\newcommand{\T}{\ensuremath\mathsf{T}}

\newcommand{\tddt}{\ensuremath{\tfrac{\mathrm{d}}{\mathrm{d}t}}}

\DeclareMathOperator{\diag}{diag}

\DeclareMathOperator{\rank}{rank}

\DeclareMathOperator{\Skew}{skew}
\DeclareMathOperator{\Sym}{sym}

\DeclareMathOperator{\trace}{trace}

\newcommand{\calH}{\mathcal{H}}

\newcommand{\calJ}{\mathcal{J}}

\newcommand{\calP}{\mathcal{P}}
\newcommand{\calR}{\mathcal{R}}

\definecolor{mycolor1}{rgb}{0.00000,0.44700,0.74100}
\definecolor{mycolor2}{rgb}{0.85000,0.32500,0.09800}
\definecolor{mycolor3}{rgb}{0.92900,0.69400,0.12500}
\definecolor{mycolor4}{rgb}{0.46600,0.67400,0.18800}
\definecolor{mycolor5}{rgb}{0.49400,0.18400,0.55600}

\newcommand{\reduce}[1]{\widetilde{#1}}

\newcommand{\state}{x}
\newcommand{\stateVar}{\state}
\newcommand{\stateDim}{n}
\newcommand{\stateDimRed}{\tilde{r}}
\newcommand{\stateDimGen}{\tilde{\stateDim}}
\newcommand{\dmdState}{\widetilde{\state}}
\newcommand{\prefix}{\Delta}

\newcommand{\inpVar}{u}

\newcommand{\inpVarDim}{m}

\newcommand{\outVar}{y}

\newcommand{\outVarDim}{p}
\newcommand{\dmdOutVar}{\widetilde{\outVar}}

\newcommand{\timeStep}{\delta_t}
\newcommand{\nrSnapshots}{M}
\newcommand{\nrSnapshotsGen}{\tilde{\nrSnapshots}}

\newcommand{\dmdM}{\mathcal{A}}
\newcommand{\dmdA}{\widetilde{A}}

\newcommand{\dmdB}{\widetilde{B}}
\newcommand{\dmdC}{\widetilde{C}}
\newcommand{\dmdD}{\widetilde{D}}

\newcommand{\dmdJ}{\widetilde{J}}
\newcommand{\dmdR}{\widetilde{R}}

\newcommand{\dmdJJ}{\widetilde{\mathcal{J}}}
\newcommand{\dmdRR}{\widetilde{\mathcal{R}}}

\newcommand{\dmdY}{\hat{Y}}
\newcommand{\dmdU}{\hat{U}}
\newcommand{\dmdV}{\hat{X}}
\newcommand{\dmdW}{\dot{\hat{X}}}

\newcommand{\podV}{\Phi}

\newcommand{\svdL}{V}
\newcommand{\svdR}{W}

\newcommand{\dataZ}{\mathcal{Z}}
\newcommand{\dataT}{\mathcal{T}}

\newcommand{\pseudo}[1]{#1^\dagger}

\newcommand{\hamiltonian}{\mathcal{H}}
\newcommand{\pHH}{\widehat{H}}
\newcommand{\pHA}{\widehat{A}}
\newcommand{\pHJ}{\widehat{J}}
\newcommand{\pHR}{\widehat{R}}
\newcommand{\pHG}{\widehat{G}}
\newcommand{\pHP}{\widehat{P}}
\newcommand{\pHD}{\widehat{D}}
\newcommand{\pHB}{\widehat{B}}
\newcommand{\pHC}{\widehat{C}}
\newcommand{\pHDsym}{\widehat{S}}
\newcommand{\pHDskew}{\widehat{N}}
\newcommand{\pHL}{\widehat{\mathcal{J}}} 
\newcommand{\pHW}{\widehat{\mathcal{R}}} 

\newcommand{\procA}{T}
\newcommand{\procB}{Z}
\newcommand{\skewProcB}{\dataZ_1}
\newcommand{\spsdProcB}{\dataZ_2}

\newcommand{\Htwo}{\calH_2}
\newcommand{\Hinf}{\calH_\infty}
\newcommand{\Ltwo}{L^2}
\newcommand{\Linf}{L^\infty}

\newcommand{\projPSD}{\calP_\succeq}

\newcommand{\norm}[1]{\left\|#1\right\|}
\newcommand{\frobnorm}[1]{\left\|#1\right\|_\mathrm{F}}

\newcommand{\bmat}[1]{\begin{bmatrix}#1\end{bmatrix}}

\newcommand{\pH}{\textsf{pH}\xspace}
\newcommand{\DMD}{\textsf{DMD}\xspace}
\newcommand{\OI}{\textsf{OI}\xspace}
\newcommand{\POD}{\textsf{POD}\xspace}
\newcommand{\SVD}{\textsf{SVD}\xspace}

\newcommand{\FGM}{\textsf{FGM}\xspace}

\newcommand{\LTI}{\textsf{LTI}\xspace}

\newcommand{\pHDMD}{\textsf{pHDMD}\xspace}

\title{Port-Hamiltonian Dynamic Mode Decomposition}
\author{Riccardo Morandin${}^\dagger$ \and Jonas Nicodemus${}^\star$ \and Benjamin Unger${}^\star$}

\address{${}^{\dagger}$  Institute of Mathematics MA\,{}4-5, Technical University Berlin, Stra\ss e des 17.~Juni 136, 10623 Berlin, Germany}
\email{morandin@math.tu-berlin.de}
\address{${}^{\star}$ Stuttgart Center for Simulation Science (SC SimTech), University of Stuttgart, Universit\"{a}tsstr.~32, 70569 Stuttgart, Germany}
\email{\{jonas.nicodemus,benjamin.unger\}@simtech.uni-stuttgart.de}
\date{\today}

\begin{document}

\begin{abstract}
	We present a novel physics-informed system identification method to construct a passive linear time-invariant system. In more detail, for a given quadratic energy functional, measurements of the input, state, and output of a system in the time domain, we find a realization that approximates the data well while guaranteeing that the energy functional satisfies a dissipation inequality. To this end, we use the framework of port-Hamiltonian (pH) systems and modify the dynamic mode decomposition, respectively operator inference, to be feasible for continuous-time pH systems. We propose an iterative numerical method to solve the corresponding least-squares minimization problem. We construct an effective initialization of the algorithm by studying the least-squares problem in a weighted norm, for which we present the analytical minimum-norm solution. The efficiency of the proposed method is demonstrated with several numerical examples.
\end{abstract}

\maketitle

%
%


\section{Introduction}
\label{sec:intro}

Incorporating prior knowledge into modern learning architectures becomes increasingly important in several applications. Such knowledge-driven or physics-informed approaches \cite{KarKLPWY21,RaiPK19} exploit expert knowledge during the learning process, either by optimizing only over a suitable set of candidate functions, or by using physical information in the cost functional. In our work, we deal with data from physical systems with the goal of identifying a linear dynamical system, which we refer to as a realization that approximates the data as well as possible with respect to the Frobenius norm. In more detail, assume that we have measurements $(\inpVar(t_i),\state(t_i),\outVar(t_i))\in\R^{\inpVarDim}\times\R^{\stateDim}\times\R^{\inpVarDim}$ at time instances~$t_i$ for $i=0,\ldots,\nrSnapshots$. Then, we want to determine matrices $A\in\R^{\stateDim\times\stateDim}$, $B\in\R^{\stateDim\times\inpVarDim}$, $C\in\R^{\inpVarDim\times\stateDim}$, and $D\in\R^{\inpVarDim\times\inpVarDim}$ such that the data can be approximately recovered by the linear time-invariant system
\begin{equation}
	\label{eqn:LTI}
	\begin{aligned}
		\dot{\state} &= A\state + B\inpVar,\\
		\outVar &= C\state + D\inpVar.
	\end{aligned}
\end{equation}
Since we assume that the data is based on a physical process, we want to ensure that the realization satisfies a dissipation inequality, i.e., that the rate of change of the energy associated with the system is bounded by the externally supplied energy. We thus incorporate physical knowledge by prescribing the (quadratic) energy functional
\begin{equation}
	\label{eqn:Hamiltonian}
	\hamiltonian(\state) \vcentcolon= \tfrac{1}{2} \state^\T H \state
\end{equation}
with symmetric positive definite matrix $H\in\R^{\stateDim\times\stateDim}$. Our main goal is then to determine the matrices in~\eqref{eqn:LTI} such that any solution of~\eqref{eqn:LTI} satisfies the dissipation inequality
\begin{equation}
	\label{eqn:dissipationInequality}
	\tddt \hamiltonian(\state(t)) \leq \outVar(t)^\T \inpVar(t)
\end{equation}
for any $t$. One of the main advantages of requiring the learned model to satisfy a dissipation inequality is that whenever the model is coupled with another passive model via a 
power-conserving or dissipative interconnection, then the coupled model is also passive. Moreover, since the Hamiltonian also serves as a Lyapunov function, we are guaranteed that the identified system is stable (independent of the underlying physical system and the quality of the measurements). Our framework can thus be used to guarantee the physical behavior of coupled first-principle and purely data-inferred dynamical systems. To achieve this goal, we use the framework of \emph{port-Hamiltonian} (\pH) systems~\cite{JacZ12,SchJ14} and modify the \emph{dynamic mode decomposition} (\DMD) \cite{Sch10,TuRLBK14,KutBBP16} and \emph{operator inference} (\OI) \cite{PehW16} accordingly. Our main contributions are the following:
\begin{enumerate}
	\item Since \DMD is designed to compute a discrete-time dynamical system, we follow~\cite{KotL18, MehM19} and present a definition of a discrete-time \pH system in \cref{subsec:pH}, which is motivated from the structure-preserving time-discretization of a continuous-time \pH system. The corresponding modified \DMD optimization problem is formulated in \Cref{problem:pHDMD}.
	\item Although \Cref{problem:pHDMD} is convex and solvable, a closed-form solution formula is not immediately available. Instead, we propose an iterative method (\Cref{alg:pHDMDIteration}) combining the result of a skew-symmetric Procrustes problem \cite{DenHZ03} with a projected fast gradient method for a positive semidefinite Procrustes problem \cite{GilS18a}. 
	\item For an efficient initialization of \Cref{alg:pHDMDIteration}, we consider a weighted Frobenius norm, where we weight the problem according to the relevant information in the data, see the forthcoming \Cref{subsec:pHDMD:modifiedProblem}. The analytic minimum-norm solution of the weighted problem is then presented in \Cref{thm:projectedApproach} and used, up to some modification, as initialization for our iterative method.
\end{enumerate}

\subsection{Literature review}
The construction of a realization of the form~\eqref{eqn:LTI} from data is a well-studied subject with many popular approaches. We mention the eigensystem realization algorithm \cite{JuaP85}, the Loewner framework \cite{MayA07}, vector fitting \cite{GusS99}, neural ordinary differential equations \cite{CheRB19}, \OI \cite{PehW16}, and \DMD \cite{KutBBP16}. Introducing expert knowledge to these approaches is not a novel idea and may even be the key idea to quantifying the error between the realization and the true physical system, cf.~\cite{HilU22,HilU22b}. Exploiting expert knowledge in the identification of linear systems is reported in various applications and methods. For instance, the specific structure of mechanical systems is exploited in a vector fitting framework in \cite{WerGG21}. The authors of \cite{SchU16,SchUBG18} exploit the fact that certain wave-type phenomena can be represented with delay equations to construct accurate linear time-invariant surrogate models from data. In the context of the Navier-Stokes equation, the specific structure of the (semi-discretized) equations is exploited in an \OI approach in \cite{BenGHP20}. Extensions of \OI to Hamiltonian and Lagrangian dynamics are reported in \cite{ShaWK22,ShaK22}. Similar ideas are discussed in the context of \DMD in \cite{BadHMKB21} by restricting the discrete-time iteration matrix to specific manifolds. In essence, our method relies on a similar idea but with a specific structure of the iteration matrix not discussed in \cite{BadHMKB21}.

As detailed above, we use the framework of \pH systems to achieve our goal of identifying a system that satisfies a dissipation inequality. In contrast to our approach, most existing results on learning linear time-invariant \pH systems work in the frequency domain. Using rational interpolation, a \pH system is constructed in \cite{AntLI17} within the Loewner framework by interpolating the transfer function at the spectral zeros. Since these are typically not known a priori, the authors of \cite{BenGV20} propose first identifying an (unstructured) system and then computing the spectral zeros from this system. A parameterization of \pH systems is used in \cite{SchV20,Sch21,SchV21} to approximately mimic an $\mathcal{H}_\infty$-type cost functional. Note that these ideas can be transferred to time-domain data using similar ideas as, for instance, in \cite{PehGW17}. Instead of directly identifying a \pH structure (as we do in our contribution), the authors of \cite{CheMH19-ppt} propose to first identify an unstructured model and then find the nearest \pH systems. Let us emphasize that in contrast to the other methods, we assume knowledge of the energy functional, and instead of finding an arbitrary \pH system, our goal is to identify a system such that the dissipation inequality for this energy functional is guaranteed.

We mention that many of the identification methods discussed above are closely related to intrusive structure-preserving model order reduction methods for \pH systems. Popular approaches include the effort-and flow-constraint reduction method \cite{PolS12}, tangential interpolation \cite{WolLEK10,PolS11,GugPBV12}, generalized balancing \cite{BreMS20,BorSF21}, $\mathcal{H}_2$-and $\mathcal{H}_\infty$-optimal approaches \cite{SatS18,SchV20,MosL20}, and spectral factorization \cite{BreU22}. For further methods, we refer to the references cited in \cite[Rem.~8.2]{MehU22-ppt}.

\subsection{Organization of the manuscript}
After this introduction, we recall basic results on \DMD and \OI (cf.~\cref{subsec:DMD}) and review the theory of linear time-invariant \pH systems in \cref{subsec:pH}. In particular, we introduce a definition for a discrete-time \pH system, following the more general discussion in \cite{KotL18,MehM19}. The \emph{port-Hamiltonian dynamic mode decomposition}~(\pHDMD) problem is formulated in \cref{subsec:pHDMD:problem}, yielding a convex minimization problem. An iterative algorithm to solve the \pHDMD problem is presented in \cref{subsec:pHDMD:twoProcrustes}, whereas we discuss a clever initialization in \cref{subsec:pHDMD:modifiedProblem} by solving the \pHDMD problem in a weighted norm. The algorithm is then applied to several numerical examples in \cref{sec:numerics}.

\subsection{Notation}
We use the symbols $\N$, $\R$, $\R^n$, and $\R^{n\times m}$ to denote the positive integers, the real numbers, the set of columns vectors with $n\in\N$ real entries, and the set of $n\times m$ real matrices, respectively. For a matrix $A\in\R^{n\times m}$ we denote its transpose with $A^\T$ and its Moore-Penrose pseudoinverse with $\pseudo{A}$. A matrix $A\in\R^{n\times n}$ is called symmetric (skew-symmetric), if $A = A^\T$ ($A = -A^\T$). The decomposition of a squared matrix $A \in \R^{n \times n}$ into its skew-symmetric resp. symmetric part is denoted as
\begin{displaymath}
	\Sym(A)=\tfrac{1}{2}(A + A^\T), \qquad \mathrm{resp.} \qquad \Skew(A)=\tfrac{1}{2}(A - A^\T).
\end{displaymath} 
The sets of all $n\times n$ symmetric positive definite and symmetric positive semi-definite matrices with real entries are denoted with $\Spd{n}$ and $\Spsd{n}$. 
The projection $\calP_\succeq(A)$ of a matrix $A\in\R^{n\times n}$ onto $\Spsd{n}$ is given by 
\begin{displaymath}
	\calP_\succeq(A) = \Xi \diag(\Lambda_+, 0, 0) \Xi^\T,
\end{displaymath}
where $\Sym(A) = \Xi \diag(\Lambda_+, \Lambda_-, 0) \Xi^\T$ is the ordered eigendecomposition of the symmetric part of $A$. Here, $\Lambda_+$ and $\Lambda_-$ are the diagonal matrices containing the positive and negative eigenvalues, respectively.
The set of nonsingular matrices of size $n\times n$ is denoted with $\GL{n}{\R}$. 
The Stiefel manifold of $n\times r$ dimensional matrices with real entries is denoted by
\begin{equation*}
	\stiefelM{n}{r} \vcentcolon= \left\{U\in \mathbb{R}^{n\times r} \mid U^\T U = I_r\right\},
\end{equation*}
where $I_r$ denotes the $r\times r$ identity matrix.
Furthermore, given $\Omega \in \Spsd{n}$ we denote the weighted Frobenius semi-norm as 
\begin{equation}
	\label{eqn:weightedFrobeniusNorm}
	\left\|A\right\|_{\Omega} \vcentcolon= \sqrt{\trace{(A^\T\Omega A)}}.
\end{equation}

\section{Preliminaries}

\subsection{Dynamic Mode Decomposition and Operator Inference}
\label{subsec:DMD}
Assume data triples $(\inpVar_i,\state_i,\outVar_i)\in\R^{\inpVarDim}\times\R^{\stateDim}\times \R^{\outVarDim}$ (\DMD) or data quadruples $(\inpVar_i,\state_i,\dot{\state}_i,\outVar_i)\in\R^{\inpVarDim}\times\R^{\stateDim}\times\R^{\stateDim} \times \R^{\outVarDim}$ (\OI) for $i=0,\ldots,\nrSnapshots$ of a dynamical system available, which may be obtained from measurements of real phenomena or the simulation of a model. The goal of (input-output) \DMD \cite{ProBK16,AnnGS16} or \OI \cite{PehW16} is to find matrices $\dmdA,\dmdB,\dmdC,\dmdD$ of suitable size minimizing
\begin{equation}
	\label{eqn:DMDminimization}
	\min_{\dmdA,\dmdB,\dmdC,\dmdD} \sum_{i=0}^{\nrSnapshots-1} \|\prefix \state_i - \dmdA\state_i - \dmdB\inpVar_i\|_2^2 
	+ \|\outVar_i - \dmdC\state_i - \dmdD\inpVar_i\|_2^2,
\end{equation}
where $\prefix \state_i \vcentcolon= \dot{\state}_i$ if we assume derivative information of the state to be available, and $\prefix \state_i \vcentcolon= \state_{i+1}$, otherwise. The goal of solving~\eqref{eqn:DMDminimization} is to determine a linear system of the form
\begin{equation}
	\label{eqn:discreteLTI}
	\begin{aligned}
	\prefix\dmdState &= \dmdA\dmdState + \dmdB\inpVar,\\
	\dmdOutVar &= \dmdC\dmdState + \dmdD\inpVar,
	\end{aligned}
\end{equation}
that best approximates the data (in the sense of~\eqref{eqn:DMDminimization}), where~$\prefix$ denotes the differentiation operator with respect to time if we are in the continuous time setting described by \OI, and~$\prefix$ denotes the forward shift in the discrete-time setting used in \DMD.
This is conveniently achieved by introducing the matrices
\begin{align*}
	Z_0 \vcentcolon= \begin{bmatrix}
		\state_0 & \ldots & \state_{\nrSnapshots-1}\\
		\inpVar_0 & \ldots & \inpVar_{\nrSnapshots-1}
	\end{bmatrix}\in\R^{(\stateDim+\inpVarDim)\times \nrSnapshots} \qquad \text{and}\qquad
	Z_1 \vcentcolon= \begin{bmatrix}
		\prefix\state_0 & \ldots & \prefix \state_{\nrSnapshots-1}\\
		\outVar_0 & \ldots & \outVar_{\nrSnapshots-1}
	\end{bmatrix}\in\R^{(\stateDim+\outVarDim)\times \nrSnapshots}
\end{align*}
and studying the equivalent problem
\begin{equation}
	\label{eqn:DMDminimizationMatrix}
	\min_{\dmdM\in\R^{(\stateDim + \outVarDim)\times (\stateDim + \inpVarDim)}} \frobnorm{Z_1 - \dmdM Z_0},
\end{equation}
whose solution, conformably partitioned, yields the matrices $\dmdA,\dmdB,\dmdC$, and $\dmdD$. The minimum norm solution $\dmdM^\star$ of~\eqref{eqn:DMDminimizationMatrix} is given by the Moore-Penrose pseudo-inverse of the data, i.e., $\dmdM^\star = Z_1\pseudo{Z_0}$. We emphasize that this is the unique solution of~\eqref{eqn:DMDminimizationMatrix} if and only if $r\vcentcolon= \rank(Z_0) = \stateDim + \inpVarDim$. The pseudo-inverse can be computed efficiently using the \emph{singular value decomposition} (\SVD). In more detail, set $r \vcentcolon= \rank(Z_0)$ and let $Z_0 = \svdL\Sigma\svdR^\T$ with $\svdL\in\stiefelM{\stateDim+\inpVarDim}{r}$, $\Sigma\in\GL{r}{\R}$, $\svdR\in\stiefelM{\nrSnapshots}{r}$ denote the skinny \SVD of $Z_0$. Then, $\pseudo{Z_0} = \svdR\Sigma^{-1}\svdL$. Note that from a numerical point of view, we truncate singular values below a given tolerance to ensure that the inverse $\Sigma^{-1}$, respectively the associated linear system, can be computed accurately. Such a truncation is equivalent to adding a regularization term to the minimization problem \eqref{eqn:DMDminimizationMatrix}, see \cite{BenHM18} for further details.

\begin{remark}
	\label{rem:derivativeApprox}
	If no measurements of the derivative are available, then a classical finite difference approach of the form
	\begin{align}
		\label{eqn:derivativeApproximation}
		\Delta \state_i \approx \dot{\hat{\state}}_i \vcentcolon= \frac{\state_{i+1} - \state_i}{\delta_i},
	\end{align}
	may be used as an approximation, where $\delta_i$ is the time step between the snapshots. This is a common approach in the literature, see, e.g., \cite{PehW16}, for which a convergence result for $\delta_i\to0$ is available \cite[Thm.\,1]{PehW16}. Throughout this manuscript, we do not assume $\dot{\state}_i$ available, and hence work with the approximation~\eqref{eqn:derivativeApproximation}, instead.
\end{remark}

The method can be further extended to include dimensionality reduction, which is then called \emph{reduced \OI}~\cite{PehW16}. There, a low dimensional basis for the state space must is constructed. Afterwards, the data is projected onto this low-dimensional basis. Finally,~\eqref{eqn:DMDminimization} is solved for the reduced data, which results in a \emph{reduced order model}. 
The low dimensional basis generation can be achieved, for instance, by applying the \SVD to the state data, i.e., let $X=\svdL \Sigma \svdR$ denote the skinny \SVD of $X$. Then the reduced basis $\podV$ of size~$r$ can be defined as the first~$r$ columns of $\svdL$. 

\subsection{Port-Hamiltonian system}
\label{subsec:pH}
As outlined in the introduction, \pH systems are a promising modeling tool that generalize the notion of Hamiltonian systems to allow for interaction with the environment. Since the aim of \DMD is to produce a linear system, we review the \pH framework only for linear dynamics and refer the reader for a more general definition to \cite{SchJ14, MehU22-ppt}.

\begin{definition}[Linear time-invariant \pH system]
	\label{def:ltiPH}
	Assume that we have given a quadratic function $\hamiltonian(\state) = \tfrac{1}{2}\state^\T \pHH\state$, called the \emph{Hamiltonian}, with $\pHH\in\Spd{\stateDim}$ and a suitable factorization $\pHH = E^\T Q$ with $E,Q\in\R^{\stateDim\times\stateDim}$. Then the system
	\begin{equation}
		\label{eqn:pHsystem}
		\begin{bmatrix}E\dot\state \\ \outVar\end{bmatrix} =
		\begin{bmatrix}A & B \\ C & D\end{bmatrix}\begin{bmatrix}Q\state \\ \inpVar\end{bmatrix}, \qquad
		\begin{bmatrix}A & B \\ -C & -D\end{bmatrix} = \calJ-\calR,
	\end{equation}
	with $\calJ =-\calJ^\top\in\R^{(\stateDim+\inpVarDim)\times(\stateDim+\inpVarDim)}$ and $\calR\in\Spsd{\stateDim+\inpVarDim}$ is called a \emph{\pH system}.
\end{definition}

In \Cref{def:ltiPH}, the matrix $\calJ$ represents the conservation of energy, while $\calR$ describes dissipation.
Note that the specific structure requires the input and output dimension to match and directly implies that \pH systems are stable and passive, and the Hamiltonian serves as a Lyapunov function. In more detail, any \pH system together with its Hamiltonian satisfies the dissipation inequality~\eqref{eqn:dissipationInequality}. Conversely, any passive linear time-invariant system, i.e., any system for which a quadratic Hamiltonian exists such that the system with this Hamiltonian satisfies the dissipation inequality~\eqref{eqn:dissipationInequality}, has a \pH representation, cf.~\cite{BeaMV19}.
Every \pH system~\eqref{eqn:pHsystem} can be equivalently written as
\begin{subequations}
	\label{eq:pHSystem}
	\begin{align}
		\label{eq:pHControl}E\dot{\state} &= (J-R)Q\state + (G-P)\inpVar,\\
		\label{eq:pHOutput}\outVar &= (G+P)^\T Q\state + (S-N)\inpVar,
	\end{align}
\end{subequations}
with $J = -J^\top$, $N = -N^\top$, and 
\begin{equation*}
	\calR = \begin{bmatrix}R & P\\P^\T & S\end{bmatrix}\in\Spsd{\stateDim+\inpVarDim}.
\end{equation*}
For further details on the properties that are encoded within this structure we refer to \cite{BeaMXZ18}.
Multiplication of \eqref{eq:pHControl} from the left with $Q^\T$, see \cite[Sec.~4.3]{MehU22-ppt}, and introducing the matrices
\begin{align*}
	\pHJ &\vcentcolon= Q^\T JQ, & 	\pHR &\vcentcolon= Q^\T RQ, & \pHG &\vcentcolon= Q^\T G, & \pHP &\vcentcolon= Q^\T P, & \pHDsym &\vcentcolon= S, & \pHDskew &\vcentcolon= N
\end{align*}
allows us to rewrite \eqref{eq:pHSystem} as 
\begin{subequations}
	\label{eq:pHSystem2}
	\begin{align}
		\label{eq:pHControl2}\pHH\dot{\state} &= (\pHJ-\pHR)\state + (\pHG-\pHP)\inpVar,\\
		\label{eq:pHOutput2}\outVar &= (\pHG+\pHP)^\T\state + (\pHDsym-\pHDskew)\inpVar,
	\end{align}
\end{subequations}
which is linear in the system matrices. Due to the congruence transformation, we immediately conclude $\pHJ = -\pHJ^\T$, $\pHDskew = -\pHDskew^\T$ and
\begin{displaymath}
	\pHW =
	\begin{bmatrix}
		\pHR & \pHP\\
		\pHP^\T & \pHDsym
	\end{bmatrix}\in\Spsd{\stateDim+\inpVarDim}.
\end{displaymath}

\begin{remark}
	Let us emphasize that from a mathematical perspective, the choice of the energy functional yielding to a \pH system is generally not unique. Indeed, any solution of the corresponding Kalman-Yakubovich-Popov inequality can be used as the Hessian of the energy; see \cite{BeaMV19} for further details. Furthermore, recent results detail that the choice of the energy functional characterizes how easy it is to approximate the system \cite{BreU22}.
\end{remark}

To understand \pH systems in the framework of \DMD, we need to find a discrete-time analogue of \eqref{eq:pHSystem2}. Since symplectic Gauss-Legendre collocation methods are able to preserve the underlying Dirac structure of the \pH system \cite{KotL18,MehM19}, and \DMD is able to recover the original dynamics if the \DMD approximation is based on a discretization with a 1-stage Runge-Kutta method \cite{HeiU22}, we use the implicit midpoint rule for the time-discretization. 
More precisely, for a constant step size $\timeStep>0$, the continuous dynamics \eqref{eq:pHControl2} are replaced with the discrete dynamics
\begin{equation*}
	\pHH\tfrac{\state_{i+1} - \state_i}{\timeStep}
	= (\pHJ-\pHR)\tfrac{\state_{i+1}+\state_{i}}{2} + (\pHG-\pHP)\tfrac{\inpVar_{i+1}+\inpVar_i}{2},
\end{equation*}
with $\state_i \approx \state(i\timeStep)$, and $\inpVar_i\vcentcolon= \inpVar(i\timeStep)$. Note that with respect to our goal to generate a discrete-time \pH system from data, we have replaced $\inpVar((i+\tfrac{1}{2})\timeStep)$ from the implicit midpoint rule with the approximation
\begin{displaymath}
	\inpVar((i+\tfrac{1}{2})\timeStep) \approx \tfrac{\inpVar_{i+1}+\inpVar_i}{2}.
\end{displaymath}
It is natural to also replace the continuous output equation \eqref{eq:pHOutput} with its discrete counterpart using $\outVar_i \approx \outVar(i\timeStep)$. To preserve port-Hamiltonian structure, we take the average of consecutive output equations arriving at 
\begin{equation*}
	\tfrac{\outVar_{i+1}+\outVar_i}{2} = (\pHG+\pHP)^\T\tfrac{\state_{i+1}+\state_i}{2} + (\pHDsym-\pHDskew)\tfrac{\inpVar_{i+1}+\inpVar_i}{2}.
\end{equation*}
The previous analysis motivates the following definition.
\begin{definition}[discrete-time \pH system]
	Consider sequences of states $x_i\in\mathbb{R}^\stateDim$, inputs $u_i\in\mathbb{R}^\inpVarDim$, outputs $y_i\in\mathbb{R}^\inpVarDim$, and a constant time step $\timeStep>0$.
	We call a system of the form
	\begin{subequations}
		\label{eq:discretePH}
		\begin{align}
			\label{eq:discretePHstateEq}\pHH\tfrac{\state_{i+1}-\state_i}{\timeStep} &= (\pHJ-\pHR)\tfrac{\state_{i+1}+\state_i}{2} + (\pHG-\pHP)\tfrac{\inpVar_{i+1}+\inpVar_i}{2},\\
			\label{eq:discretePHoutputEQ}\tfrac{\outVar_{i+1}+\outVar_i}{2} &= (\pHG+\pHP)^\T\tfrac{\state_{i+1}+\state_i}{2} + (\pHDsym-\pHDskew)\tfrac{\inpVar_{i+1}+\inpVar_i}{2},
		\end{align}
	\end{subequations}
	a \emph{discrete-time \pH system} with quadratic Hamiltonian $\hamiltonian(\state) \vcentcolon= \tfrac{1}{2} \state^\T \pHH\state$ if and only if $\pHJ=-\pHJ^\T$, $\pHDskew=-\pHDskew^\T$, and $\pHW=\left[\begin{smallmatrix}\pHR & \pHP\\\pHP^\T & \pHDsym\end{smallmatrix}\right]\in\Spsd{\stateDim+\inpVarDim}$.
\end{definition}
\begin{remark}
	The discrete-time \pH system~\eqref{eq:discretePH} is not in the standard form of a discrete-time dynamical system. Nevertheless, assuming $\timeStep$ to be sufficiently small, defining the matrices $\pHA\vcentcolon=\pHJ-\pHR$, $\pHB\vcentcolon=\pHG-\pHP$, $\pHC\vcentcolon=(\pHG+\pHP)^\T$, and $\pHD\vcentcolon=\pHDsym-\pHDskew$, and assuming a consistent initialization of the output, i.e., $\outVar_0 =\pHC\stateVar_0+\pHD\inpVar_0$, then~\eqref{eq:discretePH} can be rewritten as
	\begin{equation}\label{eq:discreteExplicitPH}
		\begin{aligned}
		\stateVar_{i+1} &= ( \tfrac{1}{\timeStep}\pHH - \tfrac{1}{2}\pHA )^{-1}( \tfrac{1}{\timeStep}\pHH + \tfrac{1}{2}\pHA )\stateVar_i + \tfrac{1}{2} ( \tfrac{1}{\timeStep}\pHH - \tfrac{1}{2}\pHA )^{-1} \pHB ( \inpVar_{i+1} + \inpVar_i ),\\
		\outVar_i &= \pHC\stateVar_i+\pHD\inpVar_i.
		\end{aligned}
	\end{equation}
	While the equation \eqref{eq:discreteExplicitPH} can be useful, since it represents $\stateVar_{i+1}$ explicitly in terms of $\stateVar_i$ and of the output variables, the $\pH$ structure of the original system is not evident.
	Because of that, we prefer working with the system \eqref{eq:discretePH} instead.
\end{remark}
Let us emphasize that \eqref{eq:discretePH} can be equivalently written as
\begin{equation}\label{eq:discretePH2}
	\bmat{\pHH\tfrac{\state_{i+1}-\state_i}{\timeStep} \\ -\tfrac{\outVar_{i+1}+\outVar_i}{2}} = (\pHL-\pHW)\bmat{\tfrac{\state_{i+1}+\state_i}{2} \\ \tfrac{\inpVar_{i+1}+\inpVar_i}{2}},
\end{equation}
where $\pHL=-\pHL^\T$ and $\pHW\in\Spd{\stateDim+\inpVarDim}$. Defining $\Delta\hamiltonian_i \vcentcolon= \tfrac{1}{\timeStep}(\hamiltonian(\state_{i+1}) - \hamiltonian(\state_i))$ and observing 
\begin{displaymath}
	\Delta\hamiltonian_i = (\tfrac{\state_{i+1}+\state_i}{2})^\T \pHH (\tfrac{\state_{i+1}-\state_{i}}{\timeStep})
\end{displaymath}
immediately yields the following discrete version of the dissipation inequality~\eqref{eqn:dissipationInequality}.

\begin{lemma}(Discrete-time dissipation inequality)
	Any discrete-time \pH system in the form of~\eqref{eq:discretePH} satisfies the discrete-time dissipation inequality
	\begin{equation}\label{eq:discretePassivity}
		\Delta\hamiltonian_i \leq (\tfrac{\outVar_{i+1}+\outVar_i}{2})^\T (\tfrac{\inpVar_{i+1}+\inpVar_i}{2}).
	\end{equation}
\end{lemma}

\begin{remark}
	To simplify our presentation, we use a constant time step $\timeStep$ throughout this manuscript. Nevertheless, it is straightforward to use a variable time step in all what follows.
\end{remark}

\section{Port-Hamiltonian Dynamic Mode Decomposition}
\label{sec:pHDMD}

In this section, we discuss a variant of \DMD that allows to construct a continuous-time \pH system from discrete-time data.

\subsection{The port-Hamiltonian Dynamic Mode Decomposition problem}
\label{subsec:pHDMD:problem}

If the data $(\state_i,\inpVar_i,\outVar_i)$ at hand are obtained from a physical system, we expect the real system to have a \pH representation and thus want to reflect that in our time-discrete realization. Following our concise definition of a discrete-time \pH system \eqref{eq:discretePH2}, we are thus interested in solving the following problem.

\begin{problem}[Port-Hamiltonian Dynamic Mode Decomposition]
	\label{problem:pHDMD}
	Consider data points $(\state_i,\inpVar_i,\outVar_i)\in\mathbb{R}^{\stateDim}\times\mathbb{R}^{\inpVarDim}\times\mathbb{R}^\inpVarDim$ for $i=0,1,\ldots,\nrSnapshots$ and define the matrices
	\begin{align*}
		\dmdW &\vcentcolon= \tfrac{1}{\timeStep}
		\begin{bmatrix}
			\state_1-\state_0 & \ldots & \state_{\nrSnapshots} - \state_{\nrSnapshots-1}
		\end{bmatrix}\in\R^{\stateDim\times\nrSnapshots},\\
		\dmdV &\vcentcolon= \tfrac{1}{2}\begin{bmatrix}
			\state_1+\state_0 & \ldots & \state_{\nrSnapshots} + \state_{\nrSnapshots-1}
		\end{bmatrix}\in\R^{\stateDim\times\nrSnapshots},\\
		\dmdU &\vcentcolon= \tfrac{1}{2}\begin{bmatrix}
			\inpVar_1 + \inpVar_0 & \ldots & \inpVar_{\nrSnapshots} + \inpVar_{\nrSnapshots-1}
		\end{bmatrix} \in\R^{\inpVarDim\times\nrSnapshots},\\
		\dmdY &\vcentcolon= \tfrac{1}{2}\begin{bmatrix}
			\outVar_1 + \outVar_0 & \ldots & \outVar_{\nrSnapshots} + \outVar_{\nrSnapshots-1}
		\end{bmatrix}\in\R^{\inpVarDim\times\nrSnapshots}.
	\end{align*}
	Given any matrix $\pHH\in\Spd{\stateDim}$ (defining the Hamiltonian), and a reduced basis $\podV\in\R^{\stateDim\times\stateDimRed}$, find matrices $\dmdJJ,\dmdRR\in\R^{\stateDimGen\times\stateDimGen}$ with $\stateDimGen \vcentcolon= \stateDimRed + \inpVarDim$ that solve 
	\begin{equation}
		\label{eqn:pHDMDminimzation}
		\min \frobnorm{\reduce{\dataZ} - (\dmdJJ-\dmdRR)\reduce{\dataT}}\qquad\text{such that $\dmdJJ = -\dmdJJ^\top\in\R^{\stateDimGen\times\stateDimGen}$ and $\dmdRR\in\Spsd{\stateDimGen}$},
	\end{equation}
	where $\reduce{\dataZ} \vcentcolon= \bmat{\podV^\T H \podV \podV^\T\dmdW \\ -\dmdY}$ and $\reduce{\dataT} \vcentcolon= \bmat{\podV^\T\dmdV \\ \dmdU}$. If no dimension reduction is applied, then we set $\stateDimRed = \stateDim$ and $\Phi=I_{\stateDim}$.
\end{problem}

\begin{remark}
	Note that the structural properties of $\dmdJJ$ and $\dmdRR$ are preserved under the transformation with $\podV$, since it is a congruence transformation.
\end{remark}

Using standard arguments, it is easy to establish that the minimization problem~\eqref{eqn:pHDMDminimzation} is convex and solvable. Moreover, let us emphasize that by solving \Cref{problem:pHDMD}, we directly learn the matrices of a continuous-time \pH system without requiring further postprocessing. Moreover, \Cref{problem:pHDMD} includes two important special cases. First, suppose we know a-priori that no dissipation is involved. In that case, we can set $\dmdRR=0$, and the minimization problem~\eqref{eqn:pHDMDminimzation} reduces to a skew-symmetric Procrustes problem, which can be solved analytically. We refer to \cite{DenHZ03,BadHMKB21} and the forthcoming \Cref{subsec:pHDMD:twoProcrustes}. Second, if the system has no input and output, i.e., we have $\inpVarDim = 0$, then the task reduces to the identification of a dissipative Hamiltonian system, i.e., we want to solve
\begin{equation}
	\label{eqn:dissipativeHamiltonianDMD}
	\min \frobnorm{\pHH\dmdW - \begin{bmatrix}
		\dmdJ - \dmdR
	\end{bmatrix}\dmdV}\qquad\text{such that $\dmdJ = -\dmdJ^\top\in\R^{\stateDimGen\times\stateDimGen}$ and $\dmdR\in\Spsd{\stateDimGen}$}.
\end{equation}

\begin{remark}
	If instead of prescribing the Hamiltonian and thus the matrix $\pHH$, one also wants to determine $\pHH\in\Spd{\stateDim}$ from the data, then we observe that~\eqref{eqn:dissipativeHamiltonianDMD} is not a suitable formulation, since in this case~\eqref{eqn:dissipativeHamiltonianDMD} is not solvable. To see this, notice that $\Spd{\stateDim}$ is an open set. Hence, for any $\varepsilon>0$, set $\dmdJ = \dmdR = 0\in\R^{\stateDim\times\stateDim}$ and $\pHH = \smash{\tfrac{\varepsilon}{2\|\dmdW\|_{\mathrm{F}}}I_{\stateDim}}$. Then $\pHH\in\Spd{\stateDim}$, $\smash{\dmdJ} = -\smash{\dmdJ^\T}$, $\smash{\dmdR}\in\Spsd{\stateDim}$ and we have
	\begin{displaymath}
		\frobnorm{\pHH\dmdW - \begin{bmatrix}
			\dmdJ-\dmdR
		\end{bmatrix}\dmdV} = \bigg\|\tfrac{\varepsilon}{2\|\dmdW\|_{\mathrm{F}}}\dmdW\bigg\|_{\mathrm{F}} = \tfrac{\varepsilon}{2} < \varepsilon,
	\end{displaymath}
	implying that the infimum of \eqref{eqn:dissipativeHamiltonianDMD} is zero with infimizer $H = \dmdJ = \dmdR = 0$. 
\end{remark}

\subsection{An iterative algorithm for the pHDMD problem}
\label{subsec:pHDMD:twoProcrustes}

Having established the existence of a solution for \Cref{problem:pHDMD}, we will now derive a numerical algorithm to solve the minimization problem~\eqref{eqn:pHDMDminimzation}. In~\cite{GilS17}, a \emph{fast gradient method} (\FGM), cf.~\cite[p.\,90]{Nes03}, for the nearest stable matrix to a given matrix is proposed, which is formulated similarly as \Cref{problem:pHDMD}. The \FGM is an optimal first-order method for convex optimization, which means no other first-order method can converge faster while using the same first-order information.	
However, we notice that if we already have a guess for $\dmdRR$, then \Cref{problem:pHDMD} simplifies to the skew-symmetric Procrustes problem
\begin{equation}
	\label{eq:skewSymProcrustes}
	\min \frobnorm{\skewProcB - \dmdJJ\dataT} \qquad\text{such that $\dmdJJ = -\dmdJJ^\top\in\R^{\stateDimGen\times\stateDimGen}$},
\end{equation}
where $\skewProcB \vcentcolon= \dataZ + \dmdRR \dataT$. Fortunately, the solution of~\eqref{eq:skewSymProcrustes} can be computed analytically as detailed in the following theorem taken as a special case of~\cite[Lem.\,2.1]{DenHZ03}, so we propose to include this knowledge into the iterative algorithm to achieve acceleration.

\begin{theorem}
	\label{thm:skewSymmetricProcrustes}
	Let $\dataT,\skewProcB\in\mathbb{R}^{\stateDimGen\times \nrSnapshots}$ and let $\svdL\Sigma\svdR^\T = \dataT$ denote the \SVD of $\dataT$ with
	\begin{displaymath}
		\Sigma = \begin{bmatrix}
			\Sigma_1 & 0\\
			0 & 0
		\end{bmatrix},\qquad \text{with}\ \Sigma_1 = \diag(\sigma_1,\ldots,\sigma_r)\in\mathbb{R}^{r\times r}\ \text{and}\ r = \rank(\procA).
	\end{displaymath}
	Define $\Phi = [\phi_{ij}]\in\R^{r\times r}$ via $\phi_{ij} = \frac{1}{\sigma_i^2 + \sigma_j^2}\qquad\text{for}\ i,j=1,\ldots,r$. Then 
	\begin{equation}
		\label{eq:skewSymProcrustesSol}
		\dmdJJ = \svdL^\T\begin{bmatrix}
			\Phi\odot \left(2\Skew(\procB_1\Sigma_1)\right) & -\Sigma_1^{-1}\procB_3^\T\\
			\procB_3\Sigma_1^{-1} & \dmdJJ_4
		\end{bmatrix}\svdL
	\end{equation}
	is a solution of \eqref{eq:skewSymProcrustes} for any skew-symmetric matrix $\dmdJJ_4\in\mathbb{R}^{(\stateDimGen-r)\times(\stateDimGen-r)}$, where
	\begin{align*}
		\procB_1 &= \begin{bmatrix}
				I_r & 0
			\end{bmatrix}\svdL\skewProcB \svdR^\T\begin{bmatrix}
				I_r & 0
			\end{bmatrix}^\T\in\mathbb{R}^{r\times r},\\
		\procB_3 &= \begin{bmatrix}
				0 & I_{\stateDim-r}
			\end{bmatrix}\svdL\skewProcB \svdR^\T\begin{bmatrix}
				I_r & 0
			\end{bmatrix}^\T\in\mathbb{R}^{(\stateDimGen-r) \times r}.
	\end{align*}
	The solution is unique if, and only if, $\rank(\procA) = \stateDimGen$.
\end{theorem}
\begin{proof}
	The proof follows along the proof of \cite[Lem.~2.1]{DenHZ03}.
\end{proof}
On the other hand if $\dmdJJ$ is given, \Cref{problem:pHDMD} simplifies to a symmetric positive definite Procrustes problem
\begin{equation}\label{eq:spsdProcrustes}
	\min \frobnorm{\spsdProcB - \dmdRR\dataT} \qquad\text{s.t. $\dmdRR\in\Spsd{\stateDimGen}$},
\end{equation}
where $\spsdProcB=\dmdJJ \dataT - \dataZ$.
Algorithmic solutions for this problem are available \cite{GilS18a}, for instance a \FGM.
By modifying this algorithm to allow a $\spsdProcB$ depending on the optimal solution $\dmdJJ$ of the skew-symmetric Procrustes problem~\eqref{eq:skewSymProcrustes}, we arrive at \Cref{alg:pHDMDIteration}.
\begin{algorithm}[htb]
	\caption{Semi-analytical fast gradient method for \eqref{eqn:pHDMDminimzation}}
	\label{alg:pHDMDIteration}
	\begin{algorithmic}[1]
		\Statex \textbf{Input:} Data matrices $\dataZ, \dataT \in \R^{\stateDimGen \times \nrSnapshotsGen}$, initial guess for the dissipative part $\dmdRR^{(0)}\in\Spsd{\stateDim}$
		\Statex \textbf{Output:} Matrices $\dmdJJ = -\dmdJJ^\T\in\mathbb{R}^{\stateDimGen\times\stateDimGen}$ and $\dmdRR\in\Spsd{\stateDimGen}$ that minimize \eqref{eqn:pHDMDminimzation}
		\Statex
		\State $L = \sigma^2_1(\dataT)$, $q=\frac{\sigma^2_r}{L}$;
		\State $k=0$; $\alpha_1 \in (0,1)$;
		\State $Q = \dmdRR^{(0)}$;
		\While{not converged}
			\State \label{alg:pHDMD:skew1} $\dataZ_1 = \dataZ + \dmdRR^{(k)} \dataT$
			\State \label{alg:pHDMD:skew2} Solve \eqref{eq:skewSymProcrustes} for $\dataZ_1, \dataT$ according to \Cref{thm:skewSymmetricProcrustes} to obtain $\dmdJJ^{(k+1)}$ 
			\State \label{alg:pHDMD:descentDirection1} $\dataZ_2 = \dmdJJ^{(k+1)} \dataT - \dataZ$
			\State \label{alg:pHDMD:descentDirection2} $\nabla = Q \dataT \dataT^\T - \dataZ_2 \dataT$
			\State \label{alg:pHDMD:spds} $\dmdRR^{(k+1)} = \mathcal{P}_\succeq(Q - \frac{1}{L}\nabla)$
			\State \label{alg:pHDMD:update1} $\alpha_{k+1} = \tfrac{1}{2}(q-\alpha_k^2 + \sqrt{(q-\alpha_k^2)^2 + 4\alpha_k^2}),\quad\beta_k=\frac{\alpha_k(1-\alpha_k)}{\alpha_k^2+\alpha_{k+1}}$
			\State \label{alg:pHDMD:update2} $Q = \dmdRR^{(k+1)} + \beta_k (\dmdRR^{(k+1)} - \dmdRR^{(k)})$
			\State $k = k + 1$
		\EndWhile
		\State $\dmdJJ = \dmdJJ^{(k)}$, $\dmdRR = \dmdRR^{(k)}$
	\end{algorithmic}
\end{algorithm}
In more detail, within each iteration step of \Cref{alg:pHDMDIteration}, we first compute the solution of the skew-symmetric Procrustes problem (\cref{alg:pHDMD:skew1,alg:pHDMD:skew2}) and then compute the gradient with respect to the matrix $\dmdRR$ in \cref{alg:pHDMD:descentDirection1,alg:pHDMD:descentDirection2}, ignoring for the moment that we need the new iterate to be symmetric positive definite. This is achieved by projecting onto the cone of symmetric positive definite matrices in \cref{alg:pHDMD:spds}. The update is then computed as a linear combination of the current and previous iterate with the fast gradient coefficients  (cf.~\cref{alg:pHDMD:update1,alg:pHDMD:update2}).

\begin{remark}
	In general, we cannot guarantee convergence of \Cref{alg:pHDMDIteration}. Nevertheless, we can use the standard safety strategy for the fast-gradient, as, for instance, reported in \cite{GilS17}, by using a reinitialization with a standard gradient step and a backtracking line search. By doing so, classical convergence results can be obtained since our objective function is convex and solvable.
\end{remark}

We notice in our numerical experiments that the performance of the algorithm strongly depends on the initialization and may need many iterations to converge if a poor initialization is used. We thus study a particular initialization strategy in the next subsection by analyzing \Cref{problem:pHDMD}  in a weighted norm.
\subsection{A weighted pHDMD problem}
\label{subsec:pHDMD:modifiedProblem}

A different but related problem arises when the Frobenius norm of the \Cref{problem:pHDMD} is replaced with the weighted Frobenius seminorm introduced in~\eqref{eqn:weightedFrobeniusNorm}, where $\dataT^\T\dataT$ is used as the semi-definite weighting matrix.
\begin{problem}[Weighted Input-Output port-Hamiltonian Dynamic Mode Decomposition]
\label{problem:projPHDMD}
For given data $\dataZ,\dataT\in\R^{\stateDimGen\times\nrSnapshotsGen}$, solve the optimization problem 
\begin{equation}\label{eq:projOptimization}
	\min \frobnorm{\dataT^\T(\dataZ - (\dmdJJ-\dmdRR)\dataT)} \qquad\text{such that $\dmdJJ = -\dmdJJ^\top\in\R^{\stateDimGen\times\stateDimGen}$ and $\dmdRR\in\Spsd{\stateDimGen}$}.
\end{equation}
\end{problem}
Two remarks are in order, first we immediately notice that by solving the weighted problem we solve the original problem, but weighted according to the relevant information in the data.
Second, using the definition of the data matrices from \Cref{problem:pHDMD}, we observe that the $(i,i)$ entry of the matrix $\dataT^\T(\dataZ - (\dmdJJ-\dmdRR)\dataT)$ reads
\begin{align*}
	e_i^\T (\dataT^\T(\dataZ - (\dmdJJ-\dmdRR)\dataT)e_i 
	&= e_i^\T (\dmdV)^\T \pHH \dmdW e_i - e_i^\T \dmdU^\T\dmdY e_i + e_i^\T \begin{bmatrix}
		(\dmdV)^\T & \dmdU^\T
	\end{bmatrix} \dmdRR \begin{bmatrix}
		\dmdV\\\dmdU
	\end{bmatrix}e_i\\
	&= \Delta \hamiltonian_i - (\tfrac{\outVar_{i+1} + \outVar_{i}}{2})^\T(\tfrac{\inpVar_{i+1} + \inpVar_{i}}{2}) + e_i^\T \begin{bmatrix}
		(\dmdV)^\T & \dmdU^\T
	\end{bmatrix} \dmdRR \begin{bmatrix}
		\dmdV\\\dmdU
	\end{bmatrix}e_i,
\end{align*}
which resembles the power-balance equation corresponding to the discrete-time dissipation inequality~\eqref{eq:discretePassivity}. In particular, the weighted problem~\eqref{eq:projOptimization} uses the dissipation inequality as part of the cost-functional to determine the dissipative component, for which we need a good initialization.
\begin{theorem}\label{thm:projectedApproach}
	Let $\dataT,\dataZ\in\R^{\stateDimGen\times \nrSnapshotsGen}$ and let $\svdL_1\Sigma_1 \svdR_1^\T = \dataT$ denote the skinny \SVD of $\dataT$, i.e., $\svdL_1\in\stiefelM{\stateDimGen}{r}$, $\Sigma_1\in\GL{r}{\R}$, and $\svdR_1\in\stiefelM{\nrSnapshotsGen}{r}$, where $r \vcentcolon= \rank(\dataT)$. Moreover, let $\svdR_2$ complement $\svdR_1$ to an orthogonal matrix, i.e., $\begin{bmatrix}
		\svdR_1 & \svdR_2
	\end{bmatrix}\in\stiefelM{\nrSnapshotsGen}{\nrSnapshotsGen}$.
	Define
	\begin{align*}
		\widetilde{\dataZ}_1 &\vcentcolon= \Sigma_1\svdL_1^\T \dataZ \svdR_1, &
		\widetilde{\dataZ}_2 &\vcentcolon= \Sigma_1\svdL_1^\T \dataZ \svdR_2,
	\end{align*}
	and
	\begin{align}
		\label{eqn:projectedOptimization:optimalJR}
		\dmdJJ^\star \vcentcolon= \svdL_1\Sigma_1^{-1} \Skew(\widetilde{\dataZ}_1)\Sigma_1^{-1}\svdL_1^\T \qquad\text{and}\qquad
		\dmdRR^\star \vcentcolon= \svdL_1\Sigma_1^{-1}\projPSD(-\widetilde{\dataZ}_1)\Sigma_1^{-1} \svdL_1^\T.
	\end{align}
	Then, $\dmdJJ^\star$ and $\dmdRR^\star$ are the unique minimum-norm minimizers of~\eqref{eq:projOptimization} with
	\begin{equation}
		\label{eqn:projectedOptimization:OptimalValue}
		\frobnorm{\dataT^\T \dataZ - \dataT^\T(\dmdJJ^\star-\dmdRR^\star)\dataT}^2 = \frobnorm{\widetilde{\dataZ}_2}^2 + \frobnorm{\Lambda_{+}}^2,
	\end{equation}
	where $\Lambda_{+}$ is the diagonal matrix which contains the positive eigenvalues of $\Sym(\widetilde{\dataZ}_1)$
\end{theorem}

\begin{remark}
	The two terms on the right-hand side of~\eqref{eqn:projectedOptimization:OptimalValue} can be interpreted as follows. The first term is only present if we have too much data to fit, i.e., if $\dataT$ has more columns than rows. Then, in general, no linear system can perfectly capture the data, and the corresponding error contribution is $\frobnorm{\smash{\widetilde{\dataZ}_2}}^2$. On the other hand, due to the specific \pH structure, not every linear system can be written as a \pH system. This potential deviation corresponds to the second error term given by $\frobnorm{\Lambda_{+}}^2$.
\end{remark}

Before we present the proof, we first illustrate \Cref{thm:projectedApproach} with an academic toy example and need some further preliminary results that provide the best-fit in the Frobenius norm within the class of skew-symmetric and symmetric positive semi-definite matrices.

\begin{example}
	\label{ex:toyExample:projected}
	Consider $\dataZ = \begin{smallbmatrix}-1 & 2\\ \phantom{-}2 & -1/2\end{smallbmatrix}$ and $\dataT = \begin{smallbmatrix}1 & 0\\0 & 2 \end{smallbmatrix}$. 
	We immediately notice that we can choose $\svdL_1 = \svdR_1 = I_2$ and $\Sigma_1 = \dataT$. We thus obtain $\widetilde{\dataZ}_1 = \Sigma_1\svdL_1^\T\dataZ \svdR_1 = \begin{smallbmatrix}-1 & \phantom{-}2\\ \phantom{-}4& -1\end{smallbmatrix}$ and thus
	\begin{displaymath}
		\Sym(\widetilde{\dataZ}_1) = \begin{bmatrix}
			-1 & \phantom{-}3\\
			\phantom{-}3 & -1
		\end{bmatrix} \qquad\text{and}\qquad
		\Skew(\widetilde{\dataZ}_1) = \begin{bmatrix}
			0 & -1\\
			1 & \phantom{-}0
		\end{bmatrix}.
	\end{displaymath}
	Then, the ordered eigendecomposition of $\widetilde{\dataZ}_{\mathrm{sym}}$ is given by,
	\begin{displaymath}
		\widetilde{\dataZ}_{\mathrm{sym}} = \Xi\Lambda \Xi^\T = 
		\begin{bmatrix}
			\frac{\sqrt{2}}{2} & \phantom{-}\frac{\sqrt{2}}{2}\\
			\frac{\sqrt{2}}{2} & -\frac{\sqrt{2}}{2} 
		\end{bmatrix}
		\begin{bmatrix}
			2 & \phantom{-}0\\
			0 & -4
		\end{bmatrix}
		\begin{bmatrix}
			\frac{\sqrt{2}}{2} & \phantom{-}\frac{\sqrt{2}}{2}\\
			\frac{\sqrt{2}}{2} & -\frac{\sqrt{2}}{2} 
		\end{bmatrix}.
	\end{displaymath}
	Thus, the unique minimizers are given by
	\begin{equation}
		\label{eqn:exToy:minimizers}
		\dmdJJ^\star = \begin{bmatrix}
			0 & -\frac{1}{2}\\
			\frac{1}{2} & \phantom{-}0
		\end{bmatrix}\qquad\text{and}\qquad
		\dmdRR^\star = \begin{bmatrix}
			\phantom{-}2 & -1\\
			-1 & \phantom{-}\tfrac{1}{2}
		\end{bmatrix}.
	\end{equation}
	The minimal value of the optimization is given by
	\begin{displaymath}
		\frobnorm{\dataT^\T \dataZ - \dataT^\T(\dmdJJ^\star-\dmdRR^\star)\dataT} = \frobnorm{\begin{bmatrix}
			1 & 1\\
			1 & 1
		\end{bmatrix}} = 2 = \frobnorm{\Lambda_{+}}.
	\end{displaymath}
\end{example}

For the proof of \Cref{thm:projectedApproach} we need the following technical result, taken from \cite[Lem.\,7]{GilS17}, see also \cite{Hig88a}.
\begin{lemma}
	\label{lem:skewAndSpsdFit}
	Let $Z\in\R^{r\times r}$, then the minimization problem
	\begin{equation}
		\label{eqn:skewAndSpsdFit}
		\min \frobnorm{Z - (J - R)}\qquad\text{such that $J = -J^\T$ and $R\in \Spsd{r}$}.
	\end{equation}
	is solved by $J^\star\vcentcolon=\Skew(Z)$ and $R^\star\vcentcolon=\projPSD(-Z)$.
\end{lemma}

\begin{proof}[Proof of \Cref{thm:projectedApproach}]
	Let $\dataT = \svdL\Sigma \svdR^\T$ denote the singular value decomposition of $\dataT$, with partitioning 
\begin{subequations}
\label{eqn:projectedMinimization:notation}
\begin{align}
	\svdL &= \begin{bmatrix}
		\svdL_1 & \svdL_2
	\end{bmatrix}, &
	\Sigma &= \diag(\Sigma_1,0), &
	\svdR &= \begin{bmatrix}
		\svdR_1 & \svdR_2
	\end{bmatrix}
\end{align}
such that $\Sigma_1\in\GL{r}{\R}$, where $r = \rank(\dataT)$. In particular, we obtain $\dataT = \svdL_1\Sigma_1 \svdR_1^\T$. Let $\dmdJJ=-\dmdJJ^\T\in\R^{\stateDimGen\times\stateDimGen}$ and $\dmdRR\in\Spsd{\stateDimGen}$. Define
\begin{align}
	\begin{bmatrix}
		\dmdJJ_{11} & -\dmdJJ_{21}^\T\\
		\dmdJJ_{21} & \phantom{-}\dmdJJ_{22}
	\end{bmatrix} &\vcentcolon= \svdL^\T\dmdJJ \svdL, &
	\begin{bmatrix}
		\dmdRR_{11} & \dmdRR_{21}^\T\\
		\dmdRR_{21} & \dmdRR_{22}
	\end{bmatrix} &\vcentcolon= \svdL^\T\dmdRR \svdL &
	\begin{bmatrix}
		\dataZ_{11} & \dataZ_{12}\\
		\dataZ_{21} & \dataZ_{22}
	\end{bmatrix} &\vcentcolon= \svdL^\T \dataZ \svdR
\end{align}
\end{subequations}
Note that throughout the proof we work with congruence transformations of $\dmdJJ$ and $\dmdRR$, which preserve the skew-symmetry and the positive semi-definiteness.
We then obtain
\begin{align*}
	\frobnorm{\dataT^\T \dataZ - \dataT^\T (\dmdJJ-\dmdRR)\dataT}^2 &= \frobnorm{\Sigma \svdL^\T \dataZ \svdR - \Sigma \svdL^\T(\dmdJJ-\dmdRR)\svdL\Sigma}^2\\
	&= \frobnorm{\begin{bmatrix}
		\Sigma_1 \dataZ_{11} & \Sigma_1\dataZ_{12}\\
		0 & 0
	\end{bmatrix} - \begin{bmatrix}
		\Sigma_1(\dmdJJ_{11}-\dmdRR_{11})\Sigma_1 & 0\\
		0 & 0
	\end{bmatrix}}^2\\
	&= \frobnorm{\widetilde{\dataZ}_1 - \Sigma_1(\dmdJJ_{11}-\dmdRR_{11})\Sigma_1}^2 + \frobnorm{\widetilde{\dataZ}_2}^2.
\end{align*}
We immediately notice that $\dmdJJ_{21}, \dmdJJ_{22}, \dmdRR_{21}$, and $\dmdRR_{22}$ do not influence the objective function and can thus be chosen arbitrarily (provided that $\dmdJJ_{22} = -\dmdJJ_{22}^\T$ and $\dmdRR\in\Spsd{\stateDimGen}$). For our further construction we set them to zero, in agreement with~\eqref{eqn:projectedOptimization:optimalJR}. It thus suffices to minimize over all skew-symmetric matrices $\dmdJJ_{11}\in\R^{r\times r}$ and all $\dmdRR_{11}\in\Spsd{r}$.
	Thus, using \Cref{lem:skewAndSpsdFit}, we obtain
	\begin{align*}
		\min_{\substack{\dmdJJ = -\dmdJJ^\T,\\ \dmdRR\in\Spsd{\stateDimGen}}} \frobnorm{\dataT^\T \dataZ - \dataT^\T(\dmdJJ-\dmdRR)\dataT}^2 
		&= \min_{\substack{\dmdJJ_{11} = -\dmdJJ_{11}^\T,\\\dmdR_{11}\in\Spsd{r}}} \frobnorm{\widetilde{\dataZ}_1 - \Sigma_1(\dmdJJ_{11}-\dmdRR_{11})\Sigma_1}^2 + \frobnorm{\widetilde{\dataZ}_2}^2\\
		&= \min_{\dmdRR_{11}\in\Spsd{r}} \frobnorm{\widetilde{\dataZ}_1 -\Sigma_1(\dmdJJ_{11}^\star - \dmdRR_{11})\Sigma_1}^2 + \frobnorm{\widetilde{\dataZ}_2}^2\\
		&= \min_{\dmdRR_{11}\in\Spsd{r}} \frobnorm{\Sym(\widetilde{\dataZ}_1) + \Sigma_1\dmdRR_{11}\Sigma_1}^2 + \frobnorm{\widetilde{\dataZ}_2}^2\\
		&= \frobnorm{\Sym(\widetilde{\dataZ}_1) + \Sigma_1\dmdRR_{11}^\star\Sigma_1} + \frobnorm{\widetilde{\dataZ}_2}^2 = \frobnorm{\Lambda_{+}}^2 + \frobnorm{\widetilde{\dataZ}_2}^2.
	\end{align*}
	It remains to show that $\dmdJJ^\star$ and $\dmdRR^\star$ are the minimizers with minimal norm. To this end, let $\dmdJJ=-\dmdJJ^\T\in\R^{\stateDimGen\times\stateDimGen}$ and $\dmdRR\in\Spsd{\stateDimGen}$ be further minimizers of~\eqref{eq:projOptimization}. Then
	\begin{equation*}
		\frobnorm{\dmdJJ} = \frobnorm{\svdR^\T\dmdJJ \svdR} = \frobnorm{\begin{bmatrix}
			\dmdJJ_{11} & -\dmdJJ_{12}^\T\\
			\dmdJJ_{21} & \phantom{-}\dmdJJ_{22}
		\end{bmatrix}}.
	\end{equation*}
	\Cref{lem:skewAndSpsdFit} implies $\dmdJJ_{11} = \dmdJJ_{11}^\star$, and thus
	\begin{equation*}
		\frobnorm{\dmdJJ} = \frobnorm{\begin{bmatrix}
			\dmdJJ_{11}^\star & -\dmdJJ_{12}^\T\\
			\dmdJJ_{21} & \phantom{-}\dmdJJ_{22}
		\end{bmatrix}} \geq \frobnorm{\begin{bmatrix}
			\dmdJJ_{11}^\star & 0\\
			0 & 0
		\end{bmatrix}} = \frobnorm{\svdR\dmdJJ^\star \svdR} = \frobnorm{\dmdJJ^\star}.
	\end{equation*}
	A similar argument shows $\frobnorm{\dmdRR} \geq \frobnorm{\dmdRR^\star}$, which completes the proof.
\end{proof}

\subsection{Relation between the optimization problems}
\label{subsec:relationPhDMDvsProjected}

In this subsection we discuss the relation between the original \pHDMD optimization \Cref{problem:pHDMD} and the weighted \Cref{problem:projPHDMD} discussed in the previous section. We immediately obtain the following result, which showcases that whenever the data is sufficiently rich, then the cost functional of the projected optimization problem~\eqref{eq:projOptimization} provides an upper bound for the original minimization problem~\eqref{eqn:pHDMDminimzation}.

\begin{lemma}
	\label{lem:correlation}
	For given $\dataZ,\dataT\in\R^{\stateDimGen\times\nrSnapshotsGen}$ with $\rank(\dataT)=\stateDimGen$ there exists a constant $c > 0$ such that for every $\dmdJJ,\dmdRR\in\R^{\stateDimGen\times\stateDimGen}$ we have
	\begin{equation}\label{eq:correlation}
		\frobnorm{\dataZ - (\dmdJJ - \dmdRR)\dataT} \leq c \frobnorm{\dataT^\T \dataZ - \dataT^\T (\dmdJJ - \dmdRR)\dataT}.
	\end{equation}
\end{lemma}

\begin{proof}
	Let $\dataT = \svdL\Sigma \svdR^\T$ denote the singular value decomposition of $\dataT$. We then have $\pseudo{\dataT} = \svdR\pseudo{\Sigma}\svdL^\T$. Define  $c\vcentcolon= \frobnorm{\smash{\pseudo{\dataT}}} = \frobnorm{\smash{\pseudo{\Sigma}}}$. Using $\rank{\dataT}=\stateDimGen$, we conclude $\pseudo{\Sigma}\Sigma = I_{\stateDimGen}$. Let $\dmdJJ,\dmdRR\in\R^{\stateDimGen\times\stateDimGen}$. Then
	\begin{align*}
		\frobnorm{\dataZ - (\dmdJJ - \dmdRR)\dataT} 
		&= \frobnorm{\svdL\svdL^\T (\dataZ - (\dmdJJ - \dmdRR)\dataT)} 
		= \frobnorm{\pseudo{(\dataT^\T)} \dataT^\T (\dataZ - (\dmdJJ - \dmdRR)\dataT)}\\
		&\leq \frobnorm{\pseudo{(\dataT^\T)}} \frobnorm{\dataT^\T \dataZ - \dataT^\T (\dmdJJ - \dmdRR)\dataT},
	\end{align*}
	which completes the proof.
\end{proof}

If $\dataT$ has not full row rank, i.e., $\rank(\dataT)<\stateDimGen$, then we cannot expect to obtain a similar result, in particular if we use the minimum norm-minimizers from \Cref{thm:projectedApproach}. The main reason for this behavior is that in this case, using the notation as in~\eqref{eqn:projectedMinimization:notation}, we project out the data corresponding to $\dataZ_{21}$ and $\dataZ_{22}$. While the latter corresponds, similarly as $\dataZ_{12}$ to too much data, the contribution $\dataZ_{21}$ corresponds to the components $\dmdJJ_{21}-\dmdRR_{21}$, which are set zero in \Cref{thm:projectedApproach}. In more detail, using the notation as in~\eqref{eqn:projectedMinimization:notation}, we obtain for $\dmdJJ = -\dmdJJ^\T\in\R^{\stateDimGen\times\stateDimGen}$ and $\dmdRR\in\Spsd{\stateDimGen}$
\begin{subequations}
	\label{eqn:phDMD:reformulation}
\begin{align*}
	\frobnorm{\dataZ-(\dmdJJ-\dmdRR)\dataT} &= \frobnorm{\begin{bmatrix}
		\svdL_1^\T\\
		\svdL_2^\T
	\end{bmatrix}\dataZ\begin{bmatrix}
		\svdR_1 & \svdR_2
	\end{bmatrix} - \begin{bmatrix}
		\svdL_1^\T\\
		\svdL_2^\T
	\end{bmatrix}(\dmdJJ-\dmdRR)\begin{bmatrix}
		\svdL_1 & \svdL_2
	\end{bmatrix}\begin{bmatrix}
		\Sigma_1 & 0\\
		0 & 0
	\end{bmatrix}}\\
	&= \frobnorm{\begin{bmatrix}
		\dataZ_{11} & \dataZ_{12}\\
		\dataZ_{21} & \dataZ_{22}
	\end{bmatrix} - \begin{bmatrix}
		\dmdJJ_{11}-\dmdRR_{11} & -\dmdJJ_{21}^\T-\dmdRR_{12}^\T\\
		\dmdJJ_{21}-\dmdRR_{21} & \phantom{-}\dmdJJ_{22}-\dmdRR_{22}
	\end{bmatrix}\begin{bmatrix}
		\Sigma_1 & 0\\
		0 & 0
	\end{bmatrix}}\\
	&= \frobnorm{\begin{bmatrix}
		\dataZ_{11} - (\dmdJJ_{11}-\dmdRR_{11})\Sigma_1 & \dataZ_{12}\\
		\dataZ_{21} - (\dmdJJ_{21}-\dmdRR_{21})\Sigma_1 & \dataZ_{22}
	\end{bmatrix}}.
\end{align*}
\end{subequations}
Thus, whenever we find $\dmdJJ_{21},\dmdRR_{21}$ such that $\dataZ_{21}\Sigma_1^{-1} = \dmdJJ_{21}-\dmdRR_{21}$ we only have to ensure that $\dmdJJ$ remains skew-symmetric and $\dmdRR$ remains symmetric positive semi-definite. A simple way to achieve this is via the choice
\begin{equation}
	\label{eqn:pHDMDoptimzation:21component}
	\dmdJJ_{21}^\star \vcentcolon= \dataZ_{21}\Sigma_1^{-1}\qquad\text{and}\qquad \dmdRR_{21}^\star = 0.
\end{equation}
In view of \Cref{alg:pHDMDIteration}, where we only require an initialization for $\dmdRR$, we can directly use the result of \Cref{thm:projectedApproach}. Nevertheless, in our numerical experiments, we observe that the initialization with $\dmdJJ^\star$ and $\dmdRR^\star$ with the modification from~\eqref{eqn:pHDMDoptimzation:21component} already yields promising results such that they can be used even without a further application of \Cref{alg:pHDMDIteration}. However, in general these are not the optimal choices as the next example illustrates.

\begin{example}
	\label{ex:toyExample:projected:counterexample}
	Consider again \Cref{ex:toyExample:projected} with the minimizers of the weighted problem $\dmdJJ^\star$ and $\dmdRR^\star$ as presented in~\eqref{eqn:exToy:minimizers}.
	We obtain $\frobnorm{\dataZ - (\dmdJJ^\star-\dmdRR^\star)\dataT} = \sqrt{\frac{5}{2}}$ for the original \pHDMD minimization problem. Nevertheless, for
	\begin{displaymath}
		\dmdJJ^{(1)} = 
	\begin{bsmallmatrix}
		0 & -\frac{1}{5}\\
		\frac{1}{5} & \phantom{-}0
	\end{bsmallmatrix}
	\end{displaymath}
	we obtain $\frobnorm{\dataZ - (\dmdJJ^{(1)}-\dmdRR^\star)\dataT} = \sqrt{\frac{200}{105}} < \sqrt{\frac{5}{2}}$, detailing that $\dmdJJ^\star$ is not the optimal choice.
\end{example}

\section{Numerical experiments}
\label{sec:numerics}
In this section, we demonstrate the theoretical discussion on two exemplary port-Hamiltonian systems. The first one is a Mass-Spring-Damper system taken from \cite{GugPBV12} and the second one is a linear poroelastic network model; for details see~\cite{AltMU21}.
For our numerical experiments, we stop \Cref{alg:pHDMDIteration} whenever
\begin{displaymath}
	\tfrac{\frobnorm{\dmdJJ^{(k + 1)} - \dmdJJ^{(k)}}}{\frobnorm{\dmdJJ^{(k + 1)}}} + \tfrac{\frobnorm{\dmdRR^{(k + 1)} - \dmdRR^{(k)}}}{\frobnorm{\dmdRR^{(k + 1)}}} \leq \varepsilon
\end{displaymath}
with prescribed tolerance $\varepsilon$. In our experiments, we use $\varepsilon\vcentcolon=1\mathrm{e-}10$. To report the progress during our iterative method, we introduce the relative values of the cost functional of \Cref{problem:pHDMD} and \Cref{problem:projPHDMD}, which we denote by
\begin{displaymath}
	f(\dmdJJ^{(k)},\dmdRR^{(k)}) = \tfrac{\frobnorm{\dataZ - (\dmdJJ^{(k)} - \dmdRR^{(k)})\dataT}}
	{\frobnorm{\dataZ}} \quad\text{and}\quad
	f_\dataT(\dmdJJ^{(k)}, \dmdRR^{(k)}) = \tfrac{\frobnorm{\dataT^\T \dataZ - \dataT^\T (\dmdJJ^{(k)} - \dmdRR^{(k)})\dataT}}
	{\frobnorm{\dataT^\T \dataZ}},
\end{displaymath}
respectively. Moreover, since we are working with academic toy examples, we report the $\mathcal{H}_2$ and $\mathcal{H}_\infty$ errors for the identified systems, defined as
\begin{displaymath}
	\norm{\mathcal{G}}_{\calH_2} = \left(\frac{1}{2\pi}\int_{-\infty}^{\infty} \frobnorm{\mathcal{G}(\imath\omega)}^2 \mathrm{d}\omega\right)^{\frac{1}{2}} \quad\text{and}\quad \norm{\mathcal{G}}_{\calH_\infty} = \sup_{\omega\in\R} \frobnorm{\mathcal{G}(\imath \omega)},
\end{displaymath}
with $\mathcal{G}$ is the transfer function of the error system. 
Note that these errors require access to the original linear time-invariant system, which of course is not available in practical applications.

\begin{remark}
	To have a fair comparison of our method with \OI, we use also the implicit midpoint information as defined in \Cref{problem:pHDMD} for \OI in all our numerical examples.
\end{remark}

\vspace{0.2cm}
\noindent\fbox{%
    \parbox{0.98\textwidth}{%
        The code and data used to generate the subsequent results are accessible via
		\begin{center}
			\href{http://doi.org/10.5281/zenodo.6497497}{doi:10.5281/zenodo.6497497}
		\end{center}
		under MIT Common License.
    }%
}
\vspace{0.2cm}

\subsection{SISO Mass-Spring-Damper system}\label{sec:msd}

\begin{figure}[htb]
	\begin{tikzpicture}
	\tikzstyle{spring}=[thick,decorate,decoration={zigzag,pre length=0.3cm,post
	length=0.3cm,segment length=6}]
	
	\tikzstyle{damper}=[thick,decoration={markings,  
	  mark connection node=dmp,
	  mark=at position 0.5 with 
	  {
		\node (dmp) [thick,inner sep=0pt,transform shape,rotate=-90,minimum
	width=15pt,minimum height=3pt,draw=none] {};
		\draw [thick] ($(dmp.north east)+(2pt,0)$) -- (dmp.south east) -- (dmp.south
	west) -- ($(dmp.north west)+(2pt,0)$);
		\draw [thick] ($(dmp.north)+(0,-5pt)$) -- ($(dmp.north)+(0,5pt)$);
	  }
	}, decorate]
	
	\tikzstyle{ground}=[fill,pattern=north east lines,draw=none,minimum
	width=0.75cm,minimum height=0.3cm]

	\tikzstyle{wall}=[fill,pattern=north east lines,draw=none,minimum
	width=0.3cm,minimum height=0.7cm]
	
	\node[draw,outer sep=0pt,thick] (M1) [minimum width=1.5cm, minimum height=1.5cm] {$m_1$};
	\node[draw,outer sep=0pt,thick] (M2) at (3,0) [minimum width=1.5cm, minimum height=1.5cm] {$m_2$};
	\node[draw,outer sep=0pt,thick] (M3) at (6.5,0) [minimum width=1.5cm, minimum height=1.5cm] {$m_{\frac{n}{2}}$};
	\node[left of=M3] {$\scriptstyle\cdots$};
	
	\node[wall, minimum height=1.5cm] (support) at (9,0) {};
	\draw (support.north west) -- (support.south west);
	
	\draw[spring] ($(M1.east) - (0,0.5)$) -- ($(M2.west) - (0,0.5)$) 
	node [midway,above] {$k_1$};
	\draw[damper] ($(M1.east) + (0,0.5)$) -- ($(M2.west) + (0,0.5)$)
	node [midway,above, yshift=5] {$c_1$};
	\draw[spring] ($(M2.east) - (0,0.5)$) -- ($(M3.west) - (0.5,0.5)$) 
	node [midway,above] {$k_2$};
	\draw[damper] ($(M2.east) + (0,0.5)$) -- ($(M3.west) + (-0.5,0.5)$)
	node [midway,above, yshift=5] {$c_2$};
	\draw[spring] ($(M3.east) - (0,0.5)$) -- ($(support.west) - (0.0,0.5)$) 
	node [midway,above] {$k_{\frac{n}{2}}$};
	\draw[damper] ($(M3.east) + (0,0.5)$) -- ($(support.west) + (-0.0,0.5)$)
	node [midway,above, yshift=5] {$c_{\frac{n}{2}}$};
	
	\draw[-latex] ($(M1.west) - (1,0)$) -- ($(M1.west)$) node [midway, above] {$u$};
	
	\end{tikzpicture}
	\caption{Illustration of the Mass-Spring-Damper system.}
	\label{fig:msd}
\end{figure}
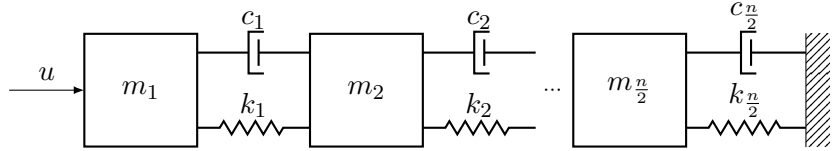
In our first experiment, we want to identify the Mass-Spring-Damper system, visualized in \Cref{fig:msd} with masses $m_i$, spring constants $k_i$ and damping constants $c_i \geq 0$ for $i=1, \ldots,\frac{n}{2}$.
A minimal realization of the Mass-Spring-Damper system, as \pH system~\eqref{eq:pHSystem2} for the order $n=6$, which corresponds to three masses, three springs and three dampers is given by 
\begin{gather*}
	\pHH =
	\begin{smallbmatrix}
		k_1 &  0 & -k_1 &  0 &  0 &  0\\
		0 &  \frac{1}{m_1} &  0 &  0 &  0 &  0\\
	   -k_1 &  0 &  k_1 + k_2 &  0 & -k_2 &  0\\
		0 &  0 &  0 &  \frac{1}{m_2} &  0 &  0\\
		0 &  0 & -k_2 &  0 &  k_2+k_3 &  0\\
		0 &  0 &  0 &  0 &  0 & \frac{1}{m_3}
	  \end{smallbmatrix}, \qquad
	\pHJ=
\begin{smallbmatrix}
  0  &  \frac{k_1}{m_1}  &  0  & -\frac{k_1}{m_2}  &  0  &  0 \\
 -\frac{k_1}{m_1}  &  0  &  \frac{k_1}{m_1}  &  0  &  0  &  0 \\
  0  & -\frac{k_1}{m_1}  &  0  &  \frac{k_1+k_2}{m_2}  &  0  & -\frac{k_2}{m_3} \\
  \frac{k_1}{m_2}  &  0  & -\frac{k_1+k_2}{m_2}  &  0  &  \frac{k_2}{m_2}  &  0 \\
  0  &  0  &  0  & -\frac{k_2}{m_2}  &  0  &  \frac{k_2+k_3}{m_3} \\
  0  &  0  &  \frac{k_2}{m_3}  &  0  & -\frac{k_2+k_3}{m_3}  &  0 
\end{smallbmatrix},\\
\pHR= \diag(0,\tfrac{c_1}{m_1^2},0,\tfrac{c_2}{m_2^2},0,\tfrac{c_3}{m_3^2}), \qquad
\pHG^\T = \bmat{
		0 & \frac{1}{m_1} & 0 & 0 & 0 & 0 \\ 
		},
\end{gather*}
and $P = 0$, $S = 0$, $N = 0$. The parameter of the masses, springs and damper are chosen as $m_i=4$, $k_i=4$, and $c_i=1$ for $i=1,\ldots,3$.
The generation of the training data $(\state_i, \inpVar_i, \outVar_i)$ for $i=0,1,\ldots,\nrSnapshots$ takes place via the implicit midpoint rule.
We collect $\nrSnapshots=100$ 
snapshots by running the simulation for 4\,\si{\second}, i.e., we set $\timeStep = \tfrac{1}{25}$.
A suitable training input is $\inpVar(t)= \exp(-\tfrac{t}{2}) \sin(t^{2})$, since it has increasing frequency to excite the model.
\begin{figure}[htb]
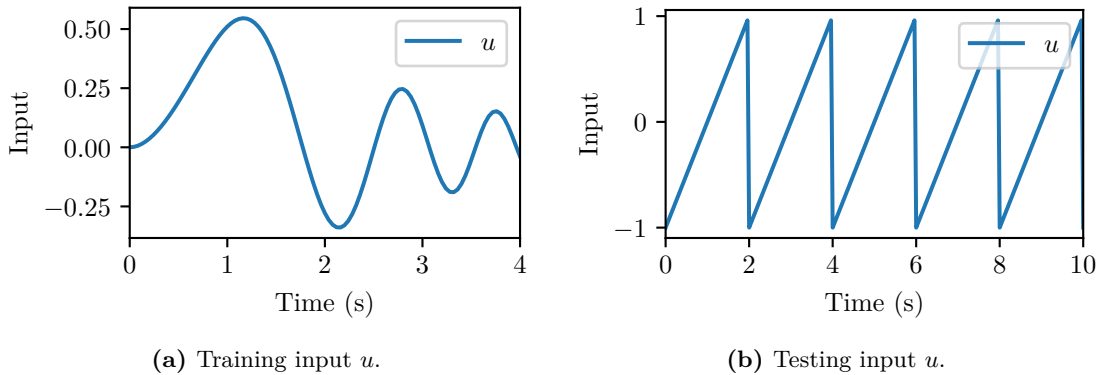

	\centering
	\begin{subfigure}{0.495\textwidth}
		\centering
		\input{figures/SISO_MSD_training_input.pgf}
		\caption{Training input $u$.}
		\label{fig:u:train}
	\end{subfigure}
	\hfill
	\begin{subfigure}{0.495\textwidth}
		\centering
		\input{figures/SISO_MSD_testing_input.pgf}
		\caption{Testing input $u$.}
		\label{fig:u:test}
	\end{subfigure}
    \caption{Training and testing input for the SISO Mass-Spring-Damper system.}
	\label{fig:u}
\end{figure}
With the collected data in place, \Cref{alg:pHDMDIteration} with $\stateDimRed = \stateDim$ and $\podV = I_{\stateDim}$ is performed using the initialization from \Cref{thm:projectedApproach}. The algorithm terminates after the first iteration. The evolution of the optimization is shown in \Cref{tb:msd_n_6_evolution}, where the values for $f_\dataT$, $f$, and the relative $\Htwo$ and relative $\Hinf$ errors are displayed in each iteration.
\begin{table}[htb]
	\centering
	\caption{Evolution of \Cref{alg:pHDMDIteration} applied for the SISO Mass-Spring-Damper system in terms of $f_\dataT$, $f$, and the relative $\Htwo$ and $\Hinf$ error over the iterations.}\label{tb:msd_n_6_evolution}
	\begin{tabular}{lcccc}
	\toprule 
	& $f_\dataT$
	& $f$
	& $\Htwo$ & $\Hinf$\\
	\midrule
	$\dmdJJ^{(0)}, \dmdRR^{(0)}$ & \num{2.61e-15} & \num{1.66e-14} & \num{2.54e-08} & \num{1.73e-11}\\
	$\dmdJJ^{(1)}, \dmdRR^{(1)}$ & \num{2.73e-14} & \num{1.38e-14} & \num{2.89e-08} & \num{1.74e-11}\\
	\bottomrule
\end{tabular}

\end{table}
We emphasize that the initialization yields a fairly good approximation of the minimizer, which is only slightly improved with \Cref{alg:pHDMDIteration}. Moreover, since in this scenario the rank of $\dataT=\stateDimGen$, according to \Cref{lem:correlation} the optimal value of \Cref{problem:projPHDMD} provides an upper bound to \Cref{problem:pHDMD},
\begin{align}
	\frobnorm{\dataZ - (\dmdJJ^{(0)} - \dmdRR^{(0)}) \dataT} \le c \frobnorm{\dataT \dataZ - \dataT (\dmdJJ^{(0)} - \dmdRR^{(0)}) \dataT} = \num{1.31e-11}
\end{align}
with $c = \num{2.49e+03}$.
The original and identified models are simulated for the testing input, shown in \Cref{fig:u:test}, for~10\,\si{\second}, i.e., on a longer time horizon than during the training.
We compare the result of our method with the result of the standard \DMD approach, where \DMD identifies a discrete-time \LTI system.
The resulting outputs $\widetilde{y}_\mathrm{DMD}$ resp. $\widetilde{y}_\mathrm{pHDMD}$ and their absolute error to the original output trajectory $y$ are presented in \Cref{fig:msd_test}.
\begin{figure}[htb]
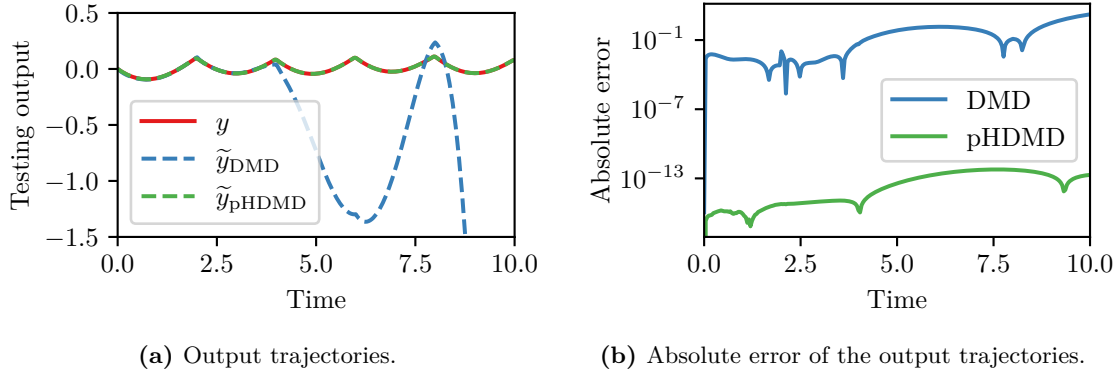

	\centering
	\begin{subfigure}{0.49\textwidth}
		\centering
		\input{figures/SISO_MSD_testing_output_with_ylims.pgf}
		\caption{Output trajectories.}
	\end{subfigure}
	\hfill
	\begin{subfigure}{0.49\textwidth}
		\centering
		\input{figures/SISO_MSD_abs_error.pgf}
		\caption{Absolute error of the output trajectories.}
	\end{subfigure}
    \caption{Output trajectories of the original and identified systems and the corresponding absolute error for the SISO Mass-Spring-Damper system. ($\delta_t=\num{4.00e-02}$)}
	\label{fig:msd_test}
\end{figure}
We notice that \DMD identifies an unstable system with five (out of 6) unstable eigenvalues. This can not happen with our method, which guarantees that the identified system is stable. However, if we decrease the stepsize, e.g., to $\delta_t=\num{1e-4}$, yielding a total of $\nrSnapshots=40000$ data points, \DMD is also able to identify a stable system. Nevertheless, the error of \DMD is still significantly larger than \pHDMD; see \Cref{fig:SISO_MSD_small_delta_abs_error}. Let us emphasize that the results of \pHDMD are, at least in parts, attributed to the generation of the data with the implicit midpoint rule. If we use a different time discretization scheme instead, for instance, the Runge-Kutta 45 (RK45) method, then the approximation quality of \pHDMD reduces several orders of magnitudes; see \Cref{fig:SISO_MSD_RK45_abs_error.pgf}. The reason for this behavior is that the data generated by RK45 does not satisfy the discrete dissipation inequality~\eqref{eq:discretePassivity}. As before, applying \DMD to data generated by evaluating the RK45 solution on a grid with the larger step size $\delta_t=\num{4.00e-02}$ yields an unstable system (cf.~\Cref{fig:SISO_MSD_RK45_abs_error.pgf}). If we use RK45 with the small step size, then \DMD produces a slightly better approximation than \pHDMD.

\begin{figure}[htb]
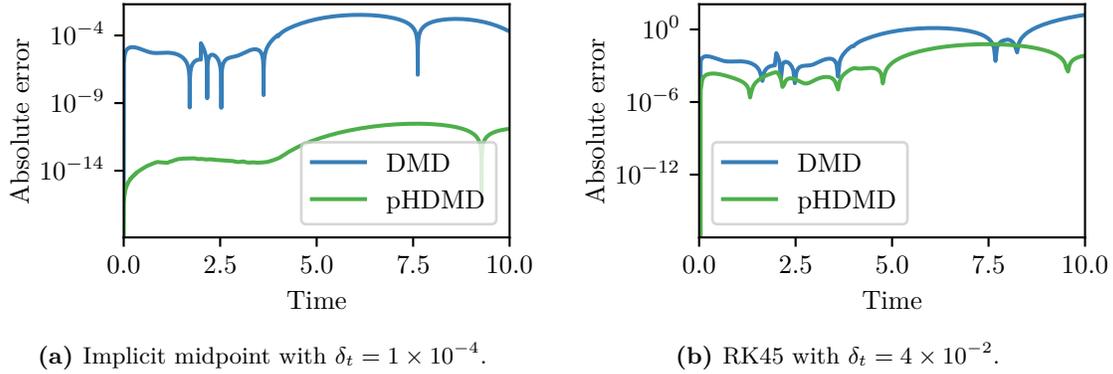

	\centering
	\begin{subfigure}{0.49\textwidth}
		\centering
		\input{figures/SISO_MSD_small_delta_abs_error.pgf}
		\caption{Implicit midpoint with $\delta_t = \num{1e-4}$.}
		\label{fig:SISO_MSD_small_delta_abs_error}
	\end{subfigure}
	\hfill
	\begin{subfigure}{0.49\textwidth}
		\centering
		\input{figures/SISO_MSD_RK45_abs_error.pgf}
		\caption{RK45 with $\delta_t = \num{4e-2}$.}
		\label{fig:SISO_MSD_RK45_abs_error.pgf}
	\end{subfigure}
    	\caption{Absolute error for \pHDMD and \DMD using data generated by the implicit midpoint rule and RK45.}
	\label{fig:label}
\end{figure}
\subsection{SISO Mass-Spring-Damper system with noisy data}
\label{sec:msd:noisy}
We repeat the experiment from \Cref{sec:msd} but add noise to the data. In particular, we add Gaussian noise with a standard deviation $s=\num{1e-4}$ to the training data ($x_i$, $y_i$) and compare the results of our method with the results of \OI, which also identifies a continuous time \LTI system. The comparison is displayed in \Cref{fig:siso-msd-noisy}.
\begin{figure}[htb]
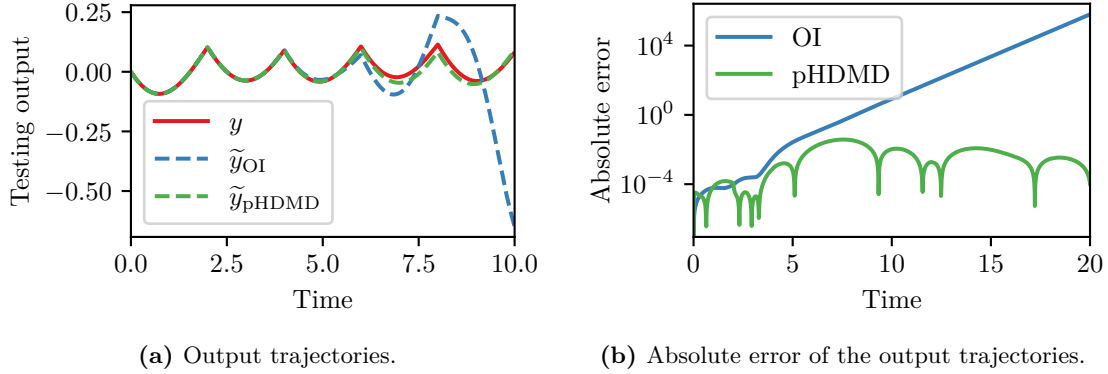

	\centering
	\begin{subfigure}{0.49\textwidth}
		\centering
		\input{figures/SISO_MSD_noisy_testing_output.pgf}
		\caption{Output trajectories.}
	\end{subfigure}
	\hfill
	\begin{subfigure}{0.49\textwidth}
		\centering
		\input{figures/SISO_MSD_noisy_long_test_abs_error.pgf}
		\caption{Absolute error of the output trajectories.}
	\end{subfigure}
	\caption{Output trajectories of the original and identified systems and the corresponding absolute error for the SISO Mass-Spring-Damper system with noisy data (standard deviation $s$ = \num{1e-4}).}
	\label{fig:siso-msd-noisy}
\end{figure}
We observe that the error of \OI increases over time while the error of \pHDMD remains small. This is again because \OI identifies an unstable system. In contrast, our method guarantees that the identified system is stable. We want to emphasize that the same argument holds for passivity

\subsection{MIMO Mass-Spring-Damper system}
In this subsection, we apply the proposed dimensionality reduction to the Mass-Spring-Damper system with increased order $\stateDim=100$. 
We also add a second input $\inpVar_2$ (and consequently also a second output), which is applied to $m_2$ the same way as $\inpVar_1$ to $m_1$ in \Cref{fig:msd}. We decrease $\delta_t$ to generate sufficiently rich data.
As training input we use $\inpVar_1(t) = \exp(-\tfrac{t}{2}) \sin(t^{2})$ and $\inpVar_2(t) = \exp(-\tfrac{t}{2}) \cos(t^{2})$.
In~\Cref{fig:mimo-msd:reduction}, we compare our method with reduced \OI and \emph{proper orthogonal decomposition} (\POD) by plotting the relative~$\Htwo$ error for different values of $\stateDimRed$. Moreover, we perturb the data by adding Gaussian noise with standard deviation $s$.
\begin{figure}[htb]
	\centering
	\input{figures/MIMO_MSD_h_norms_h2.pgf}
	\caption{$\Htwo$-norm of the error system over reduced orders $\stateDimRed$ for different noise levels with standard deviation $s$.}
	\label{fig:mimo-msd:reduction}
\end{figure}
We observe that the $\Htwo$ error of the non-intrusive reduced models is fairly close to the $\Htwo$ error of the reduced \POD model for both \pHDMD and \OI whenever the noise is sufficiently small.

\subsection{MIMO port-Hamiltonian Poroelastic Network Model}
In our final experiment, we apply our algorithm to a linear poroelasticity problem and use the \pH formulation discussed in~\cite{AltMU21}, see also \cite{BreU22}.
The model has an order of $n=980$, with two inputs and two outputs ($m=2$).
We generate $\nrSnapshots=10000$ snapshots with the same settings as before and obtain $\rank(\dataT) = 100 < \stateDimGen = 982$. Hence, we cannot expect to identify the complete system dynamics. Nevertheless, \Cref{alg:pHDMDIteration} yields matrices $\dmdJJ^\star$, $\dmdRR^\star$ with $f(\dmdJJ^\star,\dmdRR^\star) = \num{7.11e-09}$, i.e., \Cref{problem:pHDMD} is solved fairly accurately. To verify the approximation quality, we simulate the identified system with the testing input~\Cref{fig:u:test}, then we approximately recover the true output as reported in~\Cref{fig:poro:y} for the first 0.1\,\si{\second} with the corresponding absolute error displayed in~\Cref{fig:poro:e}.
\begin{figure}[htb]
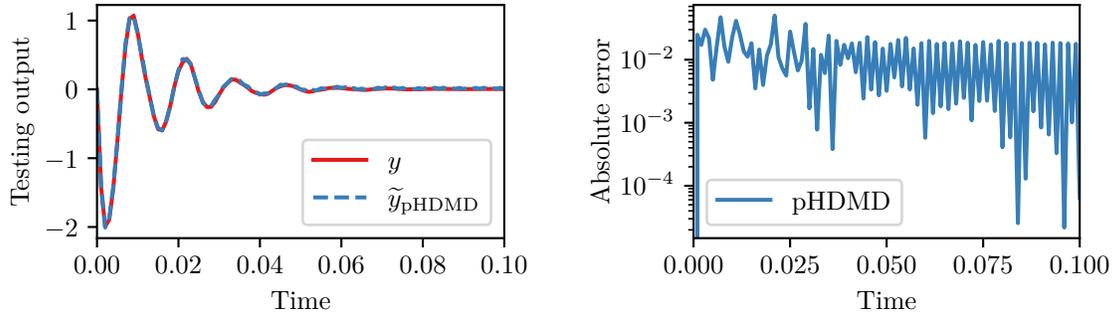

	\centering
	\begin{subfigure}{0.49\textwidth}
		\centering
		\input{figures/PORO_small_T_test_testing_output.pgf}
		\caption{Time-discretized output $y_1$ for the poroelastic network model.}
		\label{fig:poro:y}
	\end{subfigure}
	\hfill
	\begin{subfigure}{0.49\textwidth}
		\centering
		\input{figures/PORO_small_T_test_abs_error.pgf}
		\caption{Absolute error of $y_1$ for the poroelastic network model.}
		\label{fig:poro:e}
	\end{subfigure}
    \caption{Time-discretized output of the original and identified system and the corresponding absolute error for the poroelastic network model.}
	\label{fig:poro}
\end{figure}
For the full simulation until 10\,\si{\second} the identified system results in a relative $\Ltwo$ error between the output of the original system and the output of the identified system of $\num{2.95e-02}$ and a relative $\Linf$ error of $\num{2.37e-02}$. In comparison, if we use the training input, then the relative $\Ltwo$ and $\Linf$ errors are \num{5.39e-08} and \num{3.47e-07}, respectively.

\section{Conclusions}
We have developed a physics-informed system identification and dimensionality reduction algorithm that uses time-domain samples of the input, state, and output to infer a linear continuous-time dynamical system. Our algorithm is physics-informed in the sense that for a prescribed energy functional, the identified system is guaranteed to satisfy a dissipation inequality along any solution of the system. Hence, the identified system is guaranteed to be stable and passive independently of the data used for the identification. To achieve this goal, we present a generalization of dynamic mode decomposition and operator inference to port-Hamiltonian systems. The resulting method is an iterative algorithm based on a fast gradient method. For the initialization, we study a weighted problem, where we use the dominant information of the data as a weighting factor. For the weighted problem, we derive the analytical solution and detail in the numerical examples that the initialization is close to the optimum.

\subsection*{Acknowledgments}
The work of R.~Morandin is funded by the Deutsche Forschungsgemeinschaft (DFG) within the CRC/Transregio 154 \emph{Mathematical Modelling, Simulation and Optimization using the Example of Gas Networks} and the Werner-Von-Siemens Centre for Industry and Science within the project \emph{Maintenance, Repair \& Overhaul}.
J.~Nicodemus and B.~Unger acknowledge funding from the DFG under Germany's Excellence Strategy -- EXC 2075 -- 390740016 and are thankful for support by the Stuttgart Center for Simulation Science (SimTech). The authors like to thank the anonymous referees for valuable comments that significantly improved the manuscript.

\bibliographystyle{plain-doi}
\bibliography{literature}    

\end{document}